\documentclass[12pt]{article}

\usepackage{amssymb}
\usepackage{amsmath}
\usepackage{amsthm}
\usepackage{dsfont}
\usepackage[english]{babel}
\usepackage{multicol}
\usepackage{textcomp}
\usepackage{enumerate}
\usepackage{booktabs}
\usepackage{multirow}
\usepackage{array,color,colortbl} 
\providecommand{\Z}{\mathbb{Z}}
\providecommand{\tt}{\texttt}
\providecommand{\mk}{\cellcolor[gray]{.8}\tt}
\usepackage[sort&compress,numbers]{natbib}
\usepackage{todonotes}
\usepackage{authblk}
\newtheorem{theorem}{Theorem}
\newtheorem{lemma}{Lemma}
\newtheorem{corollary}{Corollary}
\newtheorem{conjecture}{Conjecture}
\newtheorem{proposition}{Proposition}

\theoremstyle{definition}

\renewcommand{\geq}{\geqslant}
\renewcommand{\leq}{\leqslant}
\renewcommand{\ge}{\geqslant}
\renewcommand{\le}{\leqslant}

\newcommand{\comp}[1]{\overline{#1}}

  \renewenvironment{thebibliography}[1]{%
    \begin{oldthebibliography}{#1}%
      \setlength{\parskip}{0.2ex}%
      \setlength{\itemsep}{0.2ex}%
  }%
  {%
    \end{oldthebibliography}%
  }
\title{Transversals in Latin arrays with many distinct symbols}

\author{Darcy~Best\thanks{Research supported by Endeavour Postgraduate Scholarship and the NSERC PGS-D.}}
\author{Kevin~Hendrey}
\author{Ian~M.~Wanless\thanks{Research supported by ARC grant DP150100506.}}
\author{Tim~E.~Wilson}
\author{David~R.~Wood\thanks{Research supported by ARC grant FT1310464.}}

\affil{\small School of Mathematical Sciences\\ 
\small Monash University\\
\small Clayton Vic 3800 Australia\\ 
\small\texttt{ 
\{darcy.best,kevin.hendrey,ian.wanless,timothy.e.wilson,david.wood\} @monash.edu
}
}
\date{}

\begin{document}

\maketitle

\begin{abstract}
  An array is {\it row-Latin} if no symbol is repeated within any row.
  An array is {\it Latin} if it and its transpose are both row-Latin.
  A {\it transversal} in an $n\times n$ array is a selection of $n$
  different symbols from different rows and different columns.  We
  prove that every $n \times n$ Latin array containing at least $(2-\sqrt{2}) n^2$
  distinct symbols has a transversal.  Also, every $n \times n$ row-Latin array
  containing at least $\frac14(5-\sqrt{5})n^2$ distinct symbols has
  a transversal.  Finally, we show by computation that every Latin
  array of order $7$ has a transversal, and we describe all smaller
  Latin arrays that have no transversal.
\end{abstract}

\section{Introduction}

This paper deals with square arrays of symbols. By an {\it entry} of such an
array $A$, we mean a triple $(i,j,A_{ij})$ where $A_{ij}$ is the symbol in
cell $(i,j)$ of $A$.
A {\it partial transversal of length $k$} in an array is a selection
of $k$ entries no pair of which agree in any of their three coordinates.
A {\it transversal} of an $n\times n$ array is a partial transversal
of length $n$ and a {\it near transversal} is a partial transversal of
length $n-1$. An array is {\it Latin} if no symbol
appears more than once in any row or column. Thus, an $n\times n$
Latin array may contain anywhere from $n$ to $n^2$ distinct
symbols. If it has just $n$ distinct symbols, then it is a {\it Latin
  square}.  Transversals of Latin squares were first studied to
construct mutually orthogonal Latin squares. Since then they have
garnered a lot of interest in their own right and lead to several
famous long-standing conjectures (see \cite{Wan11} for a survey).




For even orders $n$ there are at least $n^{n^{3/2}(1/2-o(1))}$ Latin
squares that do not have transversals \cite{CW17}. However, for
$n\times n$ Latin arrays, as
the number of distinct symbols increases, there must come a point
beyond which it becomes impossible to avoid transversals. This paper
is motivated by the question of when this threshold occurs.  Let
$\ell(n)$ be the least positive integer such that $\ell(n)\ge n$ and
every Latin array of order $n$ with at least $\ell(n)$ distinct symbols
contains a transversal. This function was introduced by Akbari and
Alipour \cite{AkbariAlipour04}, who calculated $\ell(n)$ for $n\leq 4$
and showed that $\ell(5)\geq 7$ and $\ell(2^k-2)>2^k$ for every
integer $k>2$. Counter-intuitively, every Latin square of order 5
contains a transversal, but there is a Latin array of order 5 with six
symbols and no transversal. Hence, it is not always true that increasing
the number of symbols increases the number of transversals.
Nevertheless, $\ell(n)$ is well defined since an $n\times n$ Latin
array with $n^2$ different symbols certainly has a transversal.
Akbari and Alipour put forward the following conjectures:

\begin{conjecture}\label{conj:ln}
For every integer $n\geq 3$, we have $\ell(n)\le n^2/2$.
\end{conjecture}

\begin{conjecture}\label{conj:noconst}
For every integer $c$, there exists a positive integer $n$ such that
$\ell(n)>n+c$.
\end{conjecture}

Up until this point, it was unknown whether there is some
constant $c<1$ such that $\ell(n)\leq cn^2$ for every integer
$n>1$. In Sections~\ref{sec:probabilistic}~and~\ref{sec:Latin-arrays},
we provide two independent proofs of such a result. The proof in
Section~\ref{sec:Latin-arrays} gives a better bound, but the other is
of independent interest since it demonstrates an entirely different
(probabilistic) approach. In Section~\ref{sec:small-values}, we
determine $\ell(n)$ exactly for $n\leq 7$.

On first glance, Conjecture~\ref{conj:ln} seems very generous
and that maybe $\ell(n)$ even has a linear upper bound. However, the
problem is deceptively hard, and the following observation gives some hint
as to why.

\begin{proposition}\label{prop:ln-implies-brualdi}
Let $k$ be a non-negative integer.  If $\ell(n) \leq 2kn+n-k^2-k$ for
all $n$, then every Latin square of order $n$ has a partial
transversal of length $n-k$.
\end{proposition}

\begin{proof}
 Let $L$ be any Latin square of order $n$. Let $M$ be a Latin array of
 order $n+k$, which has $L$ as the top-left $n \times n$ subarray and
 all remaining entries are new distinct symbols. The number of symbols
 in $M$ is $n+2nk+k^2\ge\ell(n+k)$, so there must be a transversal in
 $M$. This transversal hits at most $2k$ cells in the last $k$ rows or
 columns of $M$, so it must intersect the copy of $L$ in at least
 $n-k$ cells, each of which contains a different symbol.
\end{proof}

Putting $k=1$, we see that if $\ell(n)\leq 3n-2$ for all $n$, then
every Latin square has a near transversal.  This would prove a famous
conjecture attributed to Brualdi (see \cite{Wan11}). Indeed, any
linear upper bound on $\ell(n)$ would imply the existence of a
constant $c$ such that every Latin square of order $n$ has a partial
transversal of length $n-c$. The best result to date \cite{HS08} is
that every Latin square has a partial transversal of length
$n-O(\log^2n)$.

There is a broader setting in which quadratically many symbols is
known to be best possible, namely row-Latin arrays. An array is 
{\it row-Latin} if no symbol appears more than once in any row. For every
positive integer $n$, let $\ell_r(n)$ be the least positive integer
such that $\ell_r(n)\ge n$ and every $n\times n$ row-Latin array with
at least $\ell_r(n)$ distinct symbols contains a transversal. Bar{\'a}t and
Wanless \cite{BaratWanless14} showed that $\ell_r(n)>\frac12
n^2-O(n)$.  In Section~\ref{sec:Latin-arrays}, we prove that
$\ell_r(n)\leq\big\lceil\frac{1}{4}(5-\sqrt{5})n^2\big\rceil$ for
every integer $n>1$.

\section{Probabilistic Approach}\label{sec:probabilistic}
	
In this section we use probabilistic methods to prove a bound on
$\ell(n)$.  Let $\mathcal{B} = \{B_1,\ldots,B_t\}$ be a set of events
in a probability space. Usually the events $\mathcal{B}$ are called the \textit{bad events}
because the aim is for them to not occur. Define $\comp{B_i}$ to be the complement of the
event $B_i$. A graph $G$
with vertex set $\mathcal{B}$ is a \textit{lopsidependency graph} if
for all $B_i \in \mathcal{B}$ and for every subset $S$ of the complement of the closed neighbourhood
of $B_i$ in $G$,
\begin{align}\label{eq:lopsi-def}
  \mathbb{P}\Big( B_i \Big| \bigcap_{j \in S} \comp{B_j} \Big) 
  \leq \mathbb{P}(B_i).
\end{align}
 Lopsidependency graphs were introduced by Erd\H{o}s
and Spencer \cite{erdosspencer1991} and are useful because they have 
fewer edges than a naively defined dependency graph.
Intuitively, a lopsidependency graph says that the probability of an
event does not increase when conditioned on an arbitrary set of
non-adjacent events not occurring.

The Clique Lov\'asz Local Lemma by
Kolipaka, Szegedy and Xu \cite{kolipaka2012sharper} gives a condition under
which none of the bad
events occur. Specialising their formulation, we get:

\begin{lemma}\label{l:CLL}
Let $\mathcal{B} = \{B_1,\ldots,B_t\}$ be a set of events with
lopsidependency graph $G$.  Let $\{K_1,\ldots,K_s\}$ be a set of
cliques in $G$ covering all the edges, and assume $\kappa \geq \max_i|K_i|$.
Suppose that no event $B_i$ is in more than $\mu$ of the cliques 
$K_1,\dots,K_s$. If there exist $x \in (0, 1/\kappa)$ such that
$
  \mathbb{P}(B_i) \leq x \left( 1 - \kappa x \right)^{\mu-1}
$
for all $1 \leq i \leq t$, then $$\mathbb{P}\Big( \bigcap_{i=1}^t \comp{B_i} \Big) > 0.$$
\end{lemma}
	
	
We use this lemma to prove:

\begin{theorem}
Let $L$ be a Latin array of order $n$. If $L$ has at least
$({229}n^2+{27}n)/256 \approx 0.8945n^2$ distinct symbols, then $L$ has a transversal.
\end{theorem}
	
\begin{proof}
Suppose $L$ has at least 
$n^2-cn^2-dn$ distinct symbols. Let $\sigma$ be a
permutation picked uniformly at random from
the symmetric group on $\{1,2,\dots,n\}$. 
Think of $\sigma$ as choosing the cells $(i,\sigma(i))$ for
$1\le i\le n$, which might correspond to a transversal.
Define the bad events,
\begin{align*}
  \mathcal{B} = 
\{(i,j,i',j'):1\le i < i'\le n, \sigma(i)=j, \sigma(i')=j',
L_{ij}=L_{i'j'}\}.
\end{align*}
These events correspond to $\sigma$ choosing a
pairs of cells in $L$ that contain the same symbol.
To prove that a transversal exists we just need to prove
that, with positive probability, none of the bad events occur.
		
The next task is to define the lopsidependency graph which will be
used in applying Lemma~\ref{l:CLL}. Let $G$ be a graph with vertex set
$\mathcal{B}$. An edge $\{(a,b,x,y),(a',b',x',y')\}$ is in $G$ if and
only if at least two of the cells $(a,b)$, $(x,y)$, $(a',b')$ and $(x',y')$
share a row or column. This occurs only if at least one of $x=x'$,
$x=a'$, $a=x'$, $a=a'$, $y=y'$, $y=b'$, $b=y'$ or $b=b'$. 
Erd\H{o}s and Spencer \cite{erdosspencer1991} showed that $G$ is a
lopsidependency graph.

Let $\mathcal{K} = \{K_1,K_2,\ldots,K_{2n}\}$ be a set of cliques of
$G$ defined as follows. Each clique corresponds to a row or column of
$L$. An event $(a,b,x,y)$ is in a clique $K_i$ if $(a,b)$
or $(x,y)$ is in the row or column corresponding to $K_i$. Note that
$K_i \in \mathcal{K}$ is a clique because the events in $K_i$
share a row or column (the one corresponding to $K_i$) and so they are
adjacent in $G$. These cliques cover every edge of $G$ because two
events are adjacent only if they share a row or column. 

Each event in $\mathcal{B}$ corresponds to two cells in distinct rows
and columns, so each event is within exactly four cliques. Thus we
take $\mu=4$.  To find the bound $\kappa$, consider a clique
$K\in\mathcal{K}$ which, without loss of generality, corresponds to
the first row.  Each event in $K$ corresponds to two cells of $L$, one
in the first row and one not in that row.  Let $D$ be the set of cells
outside the first row that are included in some event in $K$.  Each
cell in $D$ shares a symbol with exactly one cell in the first row.
Hence $|K|=|D|$ and the cells not in $D$ contain as many distinct
symbols as $L$ does. Hence $n^2-|K|\ge n^2-cn^2-dn$, which
means that we may take $\kappa=cn^2+dn$.

Taking $x=1/(4\kappa)$, we find that to apply Lemma~\ref{l:CLL} we need
\[
\frac{1}{n(n-1)} = \mathbb{P}(B_i) \le x(1-\kappa x)^3=\frac{27}{256\kappa},
\]
which is satisfied when $c={27}/{256}$ and $d=-{27}/{256}$.
\end{proof}

\section{A Better Bound}\label{sec:Latin-arrays}

In this section, we prove a better bound on $\ell(n)$ using
non-probabilistic methods.  We start by proving results about general
square arrays, then later use these results to give bounds on the
number of symbols in transversal-free row-Latin arrays and
transversal-free Latin arrays.


We call a symbol in an array $A$ a {\it singleton} if it occurs
exactly once in $A$ and a {\it clone} otherwise. We define $R_i(A)$
and $C_j(A)$ to be the set of symbols occurring in row $i$ and column
$j$ of $A$, respectively. Let $A(i\mid j)$ denote the array formed from
$A$ by deleting row $i$ and column $j$ and let $\Psi_{ij}(A)$ be
the set of symbols that appear in $A$ and not in $A(i\mid j)$.




\begin{lemma}\label{lem:upper-bound-from-larges}
  Let $A$ be a transversal-free array of order $n$. If $A(n\mid n)$
  has a transversal and if $|R_n(A) \cup C_n(A)| \geq (k+1)n-1$, then
  $A$ has at most $$\frac12(k^2 - 2k + 2)n^2 + \frac12(3k-2)n$$
  distinct symbols.
\end{lemma}

\begin{proof}
Assume that $T$ is a near
transversal of $A$ that does not meet the last row or column and
minimises the number of symbols that it has from $R_n(A) \cup C_n(A)$.

We call a symbol \emph{large} if it appears in both $T$ and 
$R_n(A)\cup C_n(A)$ and \emph{small} otherwise.  Let $\lambda$ be the number
of large symbols.  Permute the first $n-1$ rows and columns of $A$ so
that $T$ is located along the main diagonal and all of the large
symbols of $T$ appear in the top $\lambda$ rows. For $1\le i<n$, note
that $A_{in}$ and $A_{ni}$ cannot be two different small
symbols. Otherwise, $\left(T \setminus \{(i,i,A_{ii})\}\right) \cup
\{(i,n,A_{in}),(n,i,A_{ni})\}$ would be a transversal of $A$. So there
are at most $n-1$ distinct small symbols in the last row and
column. Thus,
\begin{equation}\label{eq:bound-on-large-syms}
\lambda \geq |R_n(A) \cup C_n(A)|-(n-1) \geq (k+1)n-1-(n-1) = kn.
\end{equation}
 
We now define a subset $\Gamma$ of the entries of $A$ in which
each symbol in $A$ is represented exactly once. We populate
$\Gamma$ in three steps. First, $T \subseteq \Gamma$. Second, for
every small symbol $s$ that occurs in the last row or column, select one
such entry containing $s$ and add it to $\Gamma$. Finally, for
every symbol $s'$ in $A$ that does not appear in $T$ or in the last row
or column, select one entry with the symbol $s'$ and add it to
$\Gamma$.

We claim that if $(i,j)$ is in the top $\lambda$ rows of $A$ with $i<j<n$,
then at most one of $(i,j,A_{ij})$ and $(j,i,A_{j,i})$ can be in
$\Gamma$. Suppose otherwise, and consider 
\begin{equation}\label{e:wasP}
\big(T \setminus \big\{(i,i,A_{ii}),(j,j,A_{jj})\big\}\big) \cup
\big\{(i,j,A_{ij}),(j,i,A_{ji})\big\}.
\end{equation}
Note that the symbol $A_{ii}$ is contained in the last row or column
of $A$. By the definition of $\Gamma$, we know that $(i,j)$ and
$(j,i)$ do not have the same symbol and neither one shares a symbol
with any entry in $T$ or in the last row or column. So (\ref{e:wasP})
is a near transversal that contains fewer symbols in $R_n(A) \cup
C_n(A)$ than $T$, contradicting the choice of $T$. This implies that
within the first $\lambda$ rows and columns of $A(n\mid n)$, there are
at least
\begin{equation*}
(n-2)+(n-3)+\dots+(n-\lambda-1) = \lambda n - \frac{\lambda(\lambda+3)}{2} 
\end{equation*}
entries not contained in $\Gamma$.  Within the last row and column of $A$,
there are at most $n-1$ entries in $\Gamma$ (all containing small
symbols), so at least $n$ entries are not in $\Gamma$.  Thus,
the number of distinct symbols in $A$ is
\begin{equation}\label{eq:gamma-upper-1} 
|\Gamma| \leq n^2 - \left(\lambda n - \frac{\lambda(\lambda+3)}{2}\right) - n
= \frac12\lambda^2-\left(n-\frac32\right)\lambda+n(n-1).
\end{equation}
This quadratic in $\lambda$ decreases weakly on the integer points
in the interval $kn\le\lambda\le n-1$.
Given (\ref{eq:bound-on-large-syms}), we may substitute $\lambda=kn$
into (\ref{eq:gamma-upper-1}) to get the desired result.
\end{proof}


\begin{lemma}\label{lem:good-col}
 Let $A$ be an $n\times n$ array with $\beta n^2$ distinct symbols. If
 there are $d \geq 1$ clones in row $i$, then there is some clone
 $A_{ij}$ such that 
 $$|R_i(A) \cup C_j(A)| \geq |R_i(A)|+\frac{\beta n^2-(n-d)(n-1)-|R_i(A)|}{d}.$$
\end{lemma}

\begin{proof}
We will endeavour to find a column $j$ such that $|C_j(A) \setminus R_i(A)|$
is large. Without loss of generality, assume that the rightmost $d$
columns of row $i$ contain clones. First, remove all
occurrences of the symbols in $R_i(A)$ from the array. Now, arbitrarily select
a representative entry for each of the remaining symbols in the
array. Note that there are no representatives in row $i$ and so there
are at most $(n-d)(n-1)$ representatives in the first $n-d$
columns. Of the original $\beta n^2$ symbols, at least
$\beta n^2-(n-d)(n-1)-|R_i(A)|$ must have their representative 
in the last $d$ columns. By the pigeon-hole principle, the desired clone $A_{ij}$
occurs in one of the last $d$ columns.
\end{proof}

Let $\mathcal{A}$ be some class of square arrays of symbols that has
the following two properties: (i)~if any row and column of an array in
$\mathcal{A}$ is deleted, the resulting array is in $\mathcal{A}$ and
(ii)~if in one entry of the array, the symbol is changed to a new
symbol that appears nowhere else in the array, then the resulting
array is in $\mathcal{A}$. Note that $\mathcal{L}$, the set of all
Latin arrays, and $\mathcal{R}$, the set of all row-Latin arrays, both
satisfy the requirements listed.

Let $\frac12\le \alpha\le 1$. Define $\mathcal{M}_{\mathcal{A}}(\alpha)$
to be the set of transversal-free arrays in $\mathcal{A}$ whose
ratio of number of distinct symbols to cells is at least $\alpha$.
Suppose that $\mathcal{M}_{\mathcal{A}}(\alpha)$ is non-empty.
Define $\mathcal{M}^*_{\mathcal{A}}(\alpha) \subseteq
\mathcal{M}_{\mathcal{A}}(\alpha)$ by the rule that if $A \in
\mathcal{M}^*_{\mathcal{A}}(\alpha)$, then no array in
$\mathcal{M}_{\mathcal{A}}(\alpha)$ has an order smaller than $A$ and
no array in $\mathcal{M}_{\mathcal{A}}(\alpha)$ of the same order as
$A$ contains more distinct symbols than $A$. 
For example, 
both $\mathcal{M}^*_{\mathcal{L}}(1/2)$  and $\mathcal{M}^*_{\mathcal{R}}(1/2)$ 
consist solely of the Latin squares of order $2$.
For the remainder of the section, we will bound the number of symbols
in arrays by examining properties of the arrays in
$\mathcal{M}^*_{\mathcal{A}}(\alpha)$.

\begin{lemma}\label{lem:uniquemin}
Let $A \in \mathcal{M}^*_{\mathcal{A}}(\alpha)$ be an array of order $n$.
If $A_{ij}$ is a singleton, then $|\Psi_{ij}(A)|>\alpha(2n-1)$ and
$R_i(A)$ (resp., $C_j(A)$) contains more than $(2\alpha-1)n$ symbols that 
appear only in row $i$ (resp., column $j$) of $A$.
\end{lemma}

\begin{proof}
Any array of order $1$ has a transversal, so $n \geq 2$. There is no 
transversal $T$ of $A(i\mid j)$, or else $T\cup\{(i,j,A_{ij})\}$ would be
a transversal of $A$. As 
$A\in\mathcal{M}^*_{\mathcal{A}}(\alpha)$, we have that 
$A(i\mid j) \not\in\mathcal{M}_{\mathcal{A}}(\alpha)$, 
so the number of distinct symbols in $A(i\mid j)$ is
strictly less than $\alpha(n-1)^2$. Thus,
$$|\Psi_{ij}(A)|>\alpha n^2-\alpha(n-1)^2= \alpha(2n-1).$$
At most $n-1$ of the symbols in $\Psi_{ij}(A) \setminus \{A_{ij}\}$
appear in $C_j(A)$, so at least 
$$|\Psi_{ij}(A)|-(n-1)>\alpha(2n-1)-(n-1) \geq (2\alpha-1)n$$ 
symbols appear in row $i$ and nowhere else in $A$. 
A similar argument applies to $C_j(A)$.
\end{proof}

\begin{lemma}\label{lem:clone-near-trans-implies-few-syms}
Let $A \in \mathcal{M}^*_{\mathcal{A}}(\alpha)$ be an array of order
$n$.  If $A_{ij}$ is a clone and $|R_i(A) \cup C_j(A)|\geq(k+1)n-1$, 
then $A$ has at most
$$\frac12(k^2 - 2k + 2)n^2 + \frac12(3k-2)n$$
distinct symbols.
\end{lemma}

\begin{proof}
  Without loss of generality, $i=j=n$.  Create $A'$ by changing the
  symbol in the $(n,n)$ cell of $A$ to a symbol that did not
  previously appear in $A$. Since $A_{ij}$ is a clone in $A$, we know
  that $A'$ contains strictly more symbols than $A$. Since $A \in
  \mathcal{M}^*_{\mathcal{A}}(\alpha)$, we conclude that $A'$ has a
  transversal, although $A$ does not. Hence there is a near
  transversal of $A$ that does not meet row $n$ or column $n$. By
  applying Lemma~\ref{lem:upper-bound-from-larges}, the result
  follows.
\end{proof}

In the best case, Lemma~\ref{lem:clone-near-trans-implies-few-syms}
falls just short of proving Conjecture~\ref{conj:ln}.

\begin{corollary}\label{cor:best-case}
 Let $A \in \mathcal{M}^*_{\mathcal{A}}(\alpha)$ be an array of order
 $n$.  If $A_{ij}$ is a clone and $|R_i(A) \cup C_j(A)| = 2n-1$, then
 $A$ has at most $(n^2+n)/2$ distinct symbols.
\end{corollary}

Lemmas~\ref{lem:good-col}, \ref{lem:uniquemin} and
\ref{lem:clone-near-trans-implies-few-syms} form the main framework
needed to bound the number of symbols. We will utilise
Lemmas~\ref{lem:good-col} and \ref{lem:uniquemin} in different ways to
find an entry $(i,j,k)$ where $k$ is a clone and row $i$ and column
$j$ contain many different symbols. We then apply
Lemma~\ref{lem:clone-near-trans-implies-few-syms} to bound the number
of symbols overall. The following subsections concentrate on specific
classes for~$\mathcal{A}$.

\subsection{Row-Latin Arrays}

In this subsection, we consider $\mathcal{A}=\mathcal{R}$, the set
of row-Latin arrays.

\begin{lemma}\label{lem:rowLatin-clone}
  Let $M \in \mathcal{M}^*_{\mathcal{R}}(\alpha)$ be a row-Latin array
  of order $n$. There exists a clone $M_{ij}$ for which $|R_i(M)
  \cup C_j(M)| \geq 2\alpha n-1$.
\end{lemma}

\begin{proof}
First suppose that there is a clone $M_{ij}$ that appears in the
same column as a singleton.  By Lemma~\ref{lem:uniquemin},
$C_j(M)$ contains at least $(2\alpha - 1)n$ symbols that appear only
in $C_j(M)$. One of these symbols may be $M_{ij}$, but 
$$|R_i(M) \cup C_j(M)| 
= |R_i(M)| + |C_j(M) \setminus R_i(M)| \geq n + (2\alpha - 1)n-1 
= 2\alpha n-1,$$ 
as required.

Hence we may assume that no column contains a singleton and a
clone.  Let $d$ be the number of columns that contain clones.
 
If $d \leq n/2$, then we can find a transversal in the following
way. Let $R$ be the $n \times d$ subarray of $M$ that contains the
clones of $M$. A result of Drisko \cite{Dri98} implies that $M$
has a partial transversal of length $d$ that is wholly inside $R$.
Since this partial transversal covers all columns that contain
clones, it can trivially be extended to a transversal using
singletons.
 
So we may assume that $d > n/2$. Since each row contains $d$ clones, we
may use Lemma~\ref{lem:good-col} with $\beta\ge\alpha$
to find some clone $M_{ij}$ such that 
\[
|R_i(M) \cup C_j(M)| \geq \frac{\alpha-1}{d}n^2 + 2n - 1
> 2(\alpha-1)n+2n-1=2\alpha n-1. \qedhere
\]
\end{proof}

We now show one of our main results, that row-Latin arrays with
many symbols must have a transversal.

\begin{theorem}\label{th:rowLatin-bound}
 Let $L$ be a row-Latin array of order $n$. If $L$ has at least
 $\frac{1}{4}(5-\sqrt{5})n^2 \approx 0.6910n^2$ distinct symbols, then
 $L$ has a transversal.
\end{theorem}

\begin{proof}
Aiming for a contradiction, suppose that
$L\in\mathcal{M}_\mathcal{R}(\alpha)$ for $\alpha=(5-\sqrt{5})/4$.
Then there exists $M \in \mathcal{M}^*_{\mathcal{R}}(\alpha)$.  Let
$M$ have order $m$. 
By Lemma~\ref{lem:rowLatin-clone}, there is a clone
$M_{ij}$ such that $|R_i(M) \cup C_j(M)| \geq 2\alpha m - 1$. By
Lemma~\ref{lem:clone-near-trans-implies-few-syms}, the number of
distinct symbols in $M$ is at most
\begin{equation*}\label{eq:rowLatin-poly}
\frac{1}{2}\left((2\alpha-1)^2-2(2\alpha-1)+2\right)m^2+
\frac{1}{2}\left(3(2\alpha-1)-2\right)m
=\alpha m^2-\frac14(3\sqrt{5}-5)m.
\end{equation*}
This contradicts the fact that $M$ has at least $\alpha m^2$ distinct
symbols, and we are done.
\end{proof}

\subsection{Latin Arrays}

In this subsection, we consider $\mathcal{A}=\mathcal{L}$, the set
of Latin arrays.

We call a Latin array $L$ of order $n$ \emph{focused} if every
singleton in $L$ occurs in a row or a column that contains only
singletons and $|\Psi_{ij}(L)|=2n-1$ for some $(i,j)$ (that is,
row $i$ and column $j$ contain only singletons). We deal with
focused and unfocused arrays separately.


For focused arrays we use the following simple adaptation
of a result of Woolbright~\cite{Woolbright78}. The original proof
was for Latin squares, but it works without change for Latin arrays
(in fact for row-Latin arrays, but we do not need that).

\begin{theorem}\label{th:partial}
  Let $L$ be an $n \times n$ Latin array and $0 \leq t < n$. 
  If $(n-t)^2 > t$, then $L$ has a partial transversal of
  length $t+1$.
\end{theorem}

In the following result, recall that we assume $\alpha\ge1/2$.

\begin{lemma}\label{lem:focused}
  Let $M \in \mathcal{M}^*_{\mathcal{L}}(\alpha)$ be a Latin array of
  order $n$. If $M$ is focused, then $M$ contains at most
  $\frac{1}{8}({6-\sqrt{2}})n^2 \approx 0.5732n^2$ distinct symbols.
\end{lemma}

\begin{proof}
Let $\delta = \left\lceil(2\alpha-1)n\right\rceil$.  
Suppose $M$ has $r$ rows and $c$ columns that contain singletons.
Permute the rows and columns of $M$ so that these singletons 
occur in the top $r$ rows and leftmost $c$ columns. Since $M$ is
focused, $\min(r,c) \geq 1$ and the bottom-right
$(n-r)\times(n-c)$ subarray does not contain any singletons. Thus,
if we consider any singleton in the last row or last column, we get
$\min(r,c) \geq \delta$ by Lemma~\ref{lem:uniquemin}.

If $\alpha \geq 3/4$, then $\min(r,c) \geq n/2$ and so
$\left\{(i,n-i+1) : 1 \leq i \leq n\right\}$ is a set of cells
containing only singletons, contradicting the fact that $M$ has no
transversal. So $\alpha < 3/4$.

Let $M'$ be the
subarray formed by the last $n-\delta$ rows and columns of $M$.
Suppose that $M$ has a partial transversal of length $n-2\delta$
wholly inside $M'$.  Then this partial transversal can easily be
extended to a transversal by selecting singletons in the first
$\delta$ rows and $\delta$ columns of $M$.  By assumption $M$ has no
transversal, so applying Theorem~\ref{th:partial} to $M'$ we find that
$(\delta+1)^2\le n-2\delta-1$. Hence 
\begin{equation}\label{e:quad}
0\ge\delta^2+4\delta+2-n\ge(2\alpha-1)^2n^2+(8\alpha-5)n+2.
\end{equation}
From the discriminant of this quadratic we learn that
$32\alpha^2-48\alpha+17\ge0$. Since $\alpha<3/4$
we have $\alpha\le({6-\sqrt{2}})/{8}$.
\end{proof} 

For any $\alpha>1/2$, it is worth noting that 
(\ref{e:quad}) fails for all large $n$. So we get an asymptotic version
of Conjecture~\ref{conj:ln} holding for focused Latin arrays. We are not able
to reach such a strong conclusion for the unfocused case.

\begin{lemma}\label{lem:entrychoice}
Let $M \in \mathcal{M}^*_{\mathcal{L}}(\alpha)$ be a Latin array of order $n$. 
If $M$ is unfocused, then there exists some clone $M_{ij}$ such that
$\left|R_i(M) \cup C_j(M) \right| \geq (\alpha+1)n-1.$
\end{lemma}
\begin{proof} 
Firstly, we consider the case that $M$ has some row or column that
contains only clones. Without loss of generality, row $i$ contains
only clones. By Lemma~\ref{lem:good-col}, there is some clone
$M_{ij}$ such that $|R_i(M) \cup C_j(M)| \geq n+(\alpha n^2 - n)/{n}
=(\alpha+1)n-1$.

Secondly, we consider the case that every row and column of $M$
contains a singleton. Since $M$ is unfocused, there is some
singleton $M_{ik}$ such that
there is a clone in both row $i$ and column $k$. By
Lemma~\ref{lem:uniquemin}, we have $|\Psi_{ik}(M)|>\alpha(2n-1)$. Each
symbol in $\Psi_{ik}(M)$ appears in either $R_i(M)$ or $C_k(M)$. Also,
$M_{ik}$ appears in both $R_i(M)$ and $C_k(M)$, so without loss of
generality, $R_i(M)$ contains at least $\left(|\Psi_{ik}(M)|+1\right)/2 >
\alpha(n-1/2)+1/2\geq \alpha n$ symbols that are in $\Psi_{ik}(M)$.  Let
$M_{ij}$ be a clone in the same row as $M_{ik}$.  Except 
possibly for $M_{ij}$, none of the $n$ symbols in $C_j(M)$ are in
$\Psi_{ik}(M)$. Hence, $|R_i(M) \cup C_j(M)|\geq \alpha n+n-1$ as required.
\end{proof}

We now show a stronger result than Theorem~\ref{th:rowLatin-bound}
holds for Latin arrays.

\begin{theorem}\label{thm:main-result}
  Let $L$ be a Latin array of order $n$. If $L$ has at least
  $\left(2-\sqrt{2}\right)n^2 \approx 0.5858n^2$ distinct symbols, 
  then $L$ has a transversal.
\end{theorem}

\begin{proof}
Aiming for a contradiction, suppose that
$L\in\mathcal{M}_\mathcal{L}(\alpha)$ for $\alpha=2-\sqrt{2}$.
Then there exists $M \in \mathcal{M}^*_{\mathcal{R}}(\alpha)$.  Let
$M$ have order $m$. Note that $M$ cannot be focused, by
Lemma~\ref{lem:focused}.
So, by Lemma~\ref{lem:entrychoice}, there is a clone $M_{ij}$ such that
$\left|R_i(M) \cup C_j(M) \right| \geq (\alpha+1)m-1$.
By Lemma~\ref{lem:clone-near-trans-implies-few-syms}, the number of
distinct symbols in $M$ is at most
\begin{equation*}\label{eq:Latin-poly}
\frac{1}{2}\left(\alpha^2-2\alpha+2\right)m^2+
\frac{1}{2}\left(3\alpha-2\right)m
=\alpha m^2-\frac12(3\sqrt{2}-4)m.
\end{equation*}
This contradicts the fact that $M$ has at least $\alpha m^2$ distinct symbols,
and we are done.
\end{proof}

\section{Small Values}
\label{sec:small-values}

We now shift our attention to small values of $n$ where we can compute
$\ell(n)$ exactly. Akbari and Alipour \cite{AkbariAlipour04} 
determined $\ell(n)$ for $n \leq 4$. We extend this search to $n\le7$
and catalogue all Latin arrays of small orders with no
transversals. For $n\ge 8$, computing $\ell(n)$ seems challenging.
We will mention a couple of unsuccessful attempts to find
examples that would provide some insight.



Following \cite{EW16}, we say that two Latin arrays are {\it
  trisotopic} if one can be changed into the other by permuting rows,
permuting columns, permuting symbols and/or transposing.  The set of
all Latin arrays trisotopic to a given array is a \emph{trisotopy
  class}. The number of transversals is a trisotopy class invariant,
so to find all transversal-free Latin arrays of a given order it
suffices to consider trisotopy class representatives. However, for orders
$n>5$ it becomes difficult to construct a representative of every
trisotopy class.  The following method allows us to push our results a
couple of orders further.

Let $L$ be a transversal-free Latin array. In the first two rows of
$L$, select two entries that do not share a column or symbol (this can
always be done for $n \geq 3$). Without loss of generality, we may
assume that these two entries are $(1,1,x)$ and $(2,2,y)$. Let $L'$ be
the bottom-right $(n-2) \times (n-2)$ subarray of $L$ where all
occurrences of $x$ and $y$ are replaced with a hole (that is, a cell
with no symbol; we forbid holes from being chosen in a transversal or
partial transversal). There cannot be a partial transversal of length
$n-2$ in $L'$, otherwise the corresponding entries in $L$, together with
$(1,1,x)$ and $(2,2,y)$, would form a transversal of~$L$.

Thus, to search for transversal-free Latin arrays of order $n$, we
first build a catalogue $\mathcal{C}_{n-2}$ of trisotopy class
representatives of transversal-free partial Latin arrays of order $n-2$
with at most two holes in each row and each column. Starting with this
catalogue, we can reverse the argument above. At least one
representative of each trisotopy class of transversal-free Latin array
of order $n$ can be obtained by taking an element of
$\mathcal{C}_{n-2}$, filling its holes with $x$ and $y$, then
extending it to a Latin array of order $n$.

By the above technique we are able to give a complete catalogue of the
transversal-free trisotopy classes for orders $n \leq
7$. Table~\ref{tab:l_n} gives the value of $\ell(n)$ and the number of
trisotopy classes with a specific number of symbols.

\begin{table}\centering
\begin{tabular}{cc@{\hspace{1cm}}ccccc}
\toprule
 & ~& \multicolumn{4}{c}{Trisotopy Classes}\\
 \cmidrule{3-6}
$n$&$\ell(n)$&$n$ symbols&$n+1$ symbols&$n+2$ symbols&Total\\
\midrule
2 & 3 & 1 &  - & - &  1 \\
3 & 3 & - &  - & - &  0 \\
4 & 6 & 1 &  1 & - &  2 \\
5 & 7 & - &  2 & - &  2 \\
6 & 9 & 8 & 19 & 1 & 28 \\
7 & 7 & - &  - & - &  0\\
\bottomrule
\end{tabular}
\caption{\label{tab:l_n}Values of $\ell(n)$ and the number of 
  trisotopy classes of transversal-free Latin arrays.}
\end{table}

Representatives of the trisotopy classes of transversal-free Latin arrays
of orders $4$ and $5$ are:
\[
\left(\begin{array}{cccc}
\tt a&\tt b&\tt c&\tt d\\
\tt b&\tt c&\tt d&\tt a\\
\tt c&\tt d&\tt a&\tt b\\
\tt d&\tt a&\tt b&\tt c\\
\end{array}\right),\quad
\left(\begin{array}{cccc}
\tt a&\tt b&\tt c&\tt d\\
\tt b&\tt c&\tt a&\tt e\\
\tt c&\tt a&\tt d&\tt b\\
\tt e&\tt d&\tt b&\tt a\\
\end{array}\right),\quad
\left(\begin{array}{ccccc}
\tt a&\tt b&\tt c&\tt d&\tt e\\
\tt b&\tt c&\tt a&\tt e&\tt f\\
\tt c&\tt a&\tt b&\tt f&\tt d\\
\tt e&\tt d&\tt f&\tt c&\tt a\\
\tt d&\tt f&\tt e&\tt a&\tt b\\
\end{array}\right),\quad
\left(\begin{array}{ccccc}
\tt f&\tt b&\tt c&\tt d&\tt e\\
\tt b&\tt c&\tt a&\tt e&\tt f\\
\tt c&\tt a&\tt b&\tt f&\tt d\\
\tt e&\tt d&\tt f&\tt c&\tt a\\
\tt d&\tt f&\tt e&\tt a&\tt b\\
\end{array}\right).
\]
Note that our two representatives of order 5 differ only in their first
entry. Both can be completed to Latin squares of order 6; in the first
case this Latin square has no transversals, but in the second case it
has eight transversals.

Many of the transversal-free Latin arrays for order $6$ also turn out
to be quite similar to one another. There are exactly $28$ trisotopy
classes for $n=6$.  Previously, nine of these classes were known:
eight Latin squares and the array constructed by Akbari and
Alipour~\cite{AkbariAlipour04} by removing two rows and columns from
the elementary abelian Cayley table of order $8$. We will now describe
the $19$ transversal-free trisotopy classes of order $6$ with seven
symbols.  We will denote their representative arrays by
$L_1,L_2,\dots,L_{19}$. Let
$$L_1 = 
\left(\begin{array}{cccccc}
\mk{a}&\tt b&\tt c&\tt d&\tt e&\tt f\\
\tt b&\mk{c}&\tt a&\tt f&\tt d&\tt e\\
\tt c&\tt a&\mk{b}&\tt e&\tt f&\tt d\\
\tt d&\tt e&\tt f&\tt g&\tt b&\tt c\\
\tt f&\tt d&\tt e&\tt b&\tt g&\tt a\\
\tt e&\tt f&\tt d&\tt c&\tt a&\tt g
\end{array}\right) 
\mbox{ \ and \ }
L' =
 \left(\begin{array}{cccccc}
\tt a&\tt b&\tt c&\tt d&\tt e&\mk{f}\\
\tt c&\tt f&\tt b&\mk{e}&\tt d&\tt a\\
\tt b&\tt c&\tt e&\tt f&\mk{a}&\tt d\\
\tt d&\tt e&\mk{f}&\tt a&\tt b&\tt c\\
\mk{e}&\tt d&\tt a&\tt c&\tt f&\tt b\\
\tt f&\mk{a}&\tt d&\tt b&\tt c&\tt e
 \end{array}\right).
$$

From $L'$, we define $L_{2},\dots,L_{8}$ by changing some
entries on the main diagonal to a new symbol, $\tt g$, in the following
way. Let 
$$R' \in \big\{ \{1,2,3,4,5,6\}, \{1,2,4,5,6\},
\{1,3,4,5\}, \{1,3,6\}, \{1,4\}, \{2,3,5,6\}, \{3,4,5,6\} \big\}.$$
For all $r \in R'$, change the symbol on the main diagonal in row
$r$ of $L'$ to $\tt g$. It turns out that changing the shaded entries
in $L_1$ to $\tt g$ results in an array that is trisotopic to
$L_{2}$. Next, $L_{9}$ is obtained by changing the symbol of the shaded
entries in $L'$ to a new symbol, $\tt g$. Let
$$L_{10} =
  \left(\begin{array}{cccccc}
\tt a&\tt b&\tt c&\tt d&\tt e&\tt f\\
\tt b&\tt c&\tt g&\tt a&\tt f&\tt e\\
\tt c&\tt f&\tt d&\tt g&\tt a&\tt b\\
\tt d&\tt a&\tt f&\tt e&\tt g&\tt c\\
\tt e&\tt g&\tt a&\tt f&\tt c&\tt d\\
\tt g&\tt e&\tt b&\tt c&\tt d&\tt a
 \end{array}\right)
\mbox{ \ and \ }
L'' =
 \left(\begin{array}{cccccc}
\tt a&\tt b&\tt c&\tt d&\tt e&\tt f\\
\tt b&\tt c&\tt a&\tt e&\tt f&\tt d\\
\tt c&\tt a&\tt b&\tt f&\tt d&\tt e\\
\tt d&\tt e&\tt f&\tt a&\tt c&\tt b\\
\tt e&\tt f&\tt d&\tt c&\tt b&\tt a\\
\tt f&\tt d&\tt e&\tt b&\tt a&\tt c
 \end{array}\right).
$$
From $L_{10}$, we can either change the symbol in the $(3,3)$ cell to
$\tt e$, giving $L_{11}$, or change the symbol in the $(4,4)$ cell to $\tt b$,
giving $L_{12}$.

The array $L_{13}$ is obtained by changing the $\tt d$ in rows 2 and 3 of
$L''$ to $\tt g$, as well as changing the $\tt f$ in row 2 to $\tt
d$. Next, $L_{14}$ is obtained by changing the $\tt e$ in row 3 of
$L_{13}$ to $\tt d$.  From the Latin square $L''$, any subset of
entries that contain $\tt d$ may be changed to a new symbol, $\tt
g$. This gives rise to 5 trisotopy classes. In particular, we define
$L_{15},\dots,L_{19}$ by changing some occurrences of $\tt d$ to a new
symbol, $\tt g$, in the following way. Let
$$R'' \in\big\{\{1\},\{1,2\},\{1,2,3\},\{1,3,5\},\{1,4\}\big\}.$$ 
For all $r\in R''$, change the $\tt d$ in row $r$ of $L''$ to $\tt g$.

One can check that $L_{15},\dots,L_{19}$ are transversal-free by exhaustive
computation, but next we give a reason why they have no transversals.
The argument is in the style of the highly successful $\Delta$-lemma
(see \cite{Wan11}). 
Let $L$ be any Latin array obtained by replacing any subset of the
occurrences of $\tt d$ in $L''$ by $\tt g$.
Define functions $\rho,\nu$ to $\Z_3$ by:
\begin{align*}
&\rho(1)=\rho(4)=0,\ \rho(2)=\rho(5)=1,\ \rho(3)=\rho(6)=2,\\
&\nu({\tt a})=\nu({\tt d})=\nu({\tt g})=0,\ \nu({\tt b})=\nu({\tt e})=1,\ \nu({\tt c})=\nu({\tt f})=2.
\end{align*}
Define a function $\Delta$ from the entries of $L$ to $\Z_3$ by 
$\Delta(r,c,s)=\rho(r)+\rho(c)-\nu(s)$. Let $D$ denote
the bottom-right $3\times 3$ subsquare of $L$.
Suppose that $T$ is a transversal of $L$ and that $\bar s$ is the only symbol
in $\{{\tt a,b,\dots,g}\}$ that does not appear in $T$. Then
\begin{equation}\label{e:deltalem}
\sum_{(r,c,s)\in T}\Delta(r,c,s)=
2\sum_{i=1}^6\rho(i)-\sum_{(r,c,s)\in T}\nu(s)=\nu(\bar s).
\end{equation}
Also, if $T$ includes $x$ entries in $D$ then overall it has $2x$ entries with
symbols in $\{{\tt a,b,c}\}$, which means that $x=1$ and 
$\bar s\in\{{\tt a,b,c}\}$.
However, $\Delta(r,c,s)=0$ for all entries of $L$, except those
in $D$, where $\Delta(r,c,s)=\nu(s)$.
Hence to satisfy (\ref{e:deltalem}), the symbol in the only entry of $T$
in $D$ has to be $\bar s$, contradicting the fact that this symbol does 
not appear in $T$.

The argument we have just presented is specific to order $n=6$ and
does not seem to easily generalise to arrays of larger orders.

When performing the search for transversal-free Latin arrays of order
$n = 7$, we found $15\,611\,437$ trisotopy classes of transversal-free
partial Latin arrays of order 5 and at most two holes in each row and
column. Table~\ref{tab:5x5-partial} provides counts of the trisotopy
classes based on number of holes and number of symbols. Since none of
these arrays can extend to a Latin array of order $7$ with no
transversals, we have the following result.

\begin{theorem}\label{th:7x7-no-trans}
  Every Latin array of order $7$ has a transversal.
\end{theorem}

\begin{table}\centering
 \begin{tabular}{@{}ll@{\hskip5pt}r | lllllllllllll@{}}
 \toprule
 &&& \multicolumn{11}{c}{Number of Symbols} \\
 
 &&& 3 & 4 & 5 & 6 & 7 & 8 & 9 & 10 & 11 & 12 & 13\\
 \midrule
\parbox[t]{2mm}{\multirow{11}{*}{\rotatebox[origin=c]{90}{Number of Holes}}} & 
0  && - & - & - & 2 & - & - & - & - & - & - & - \\
& 1  && - & - & 1 & 17 & - & - & - & - & - & - & - \\
& 2  && - & - & 9 & 271 & 13 & - & - & - & - & - & - \\
& 3  && - & - & 137 & 4893 & 1179 & 61 & 5 & - & - & - & - \\
& 4  && - & - & 1484 & 54911 & 31342 & 5539 & 1906 & 462 & 62 & 4 & - \\
& 5  && - & 3 & 10686 & 341251 & 319750 & 58257 & 9823 & 1175 & 86 & 4 & - \\
& 6  && - & 19 & 48436 & 1155690 & 1420192 & 299951 & 33366 & 1953 & 56 & - & - \\
& 7  && - & 151 & 124275 & 2045859 & 2754143 & 670137 & 63480 & 2676 & 30 & - & - \\
& 8  && - & 632 & 159295 & 1720463 & 2198260 & 549316 & 43912 & 1710 & 78 & 8 & 1 \\
& 9  && - & 916 & 80609 & 557285 & 603320 & 134056 & 7120 & 148 & 7 & 1 & - \\
& 10 && 3 & 320 & 9420 & 40418 & 34218 & 6014 & 159 & 1 & - & - & - \\
 \bottomrule
 \end{tabular}
 \caption{\label{tab:5x5-partial}Counts of trisotopy classes of 
   transversal-free $5 \times 5$ partial Latin arrays, categorised by 
   number of symbols and number of holes.}
\end{table}

The approach that we used to prove Theorem~\ref{th:7x7-no-trans} is
infeasible for $n \geq 8$, although we did examine certain interesting
sets of Latin arrays of order 8.  There are 68 different
transversal-free Latin squares of order 8, up to trisotopy. We also
considered all Latin arrays which are obtained by removing one row and
one column from a Latin square of order 9. We could immediately
eliminate any square of order 9 that contains a transversal through
every entry. Latin squares that do not contain a transversal through
every entry are called \emph{confirmed bachelor squares}. The
confirmed bachelor squares of order 9 were generated for~\cite{EW12},
providing us with a set of trisotopy class representatives. None of
these squares has an order 8 transversal-free subarray. Lastly, we
searched all Latin arrays of order 8 with exactly 9 symbols where one
of the symbols appears at most 4 times. None of these were
transversal-free.  The arrays that we have checked are a tiny subset
of all Latin arrays of order 8. Without theoretical insight, it seems
hopeless to check them all. So all that we can conclude at this stage
is that $\ell(8)\ge9$.

It is known that all Latin squares of order $9$ have transversals
(see, e.g.~\cite{EW12}).  We tried, unsuccessfully, to build a
transversal-free Latin array of order $9$. We did this by removing a
row and column from Latin squares of order $10$. The squares that we
used were representatives of all trisotopy classes for which the
autoparatopy group has order $3$ or higher, as generated for
\cite{MMM07}.


The results of our investigations lead us to be skeptical that
Conjecture~\ref{conj:noconst} is true. However, proving that it is
false is likely to be extremely hard, for the reasons explained after
Proposition~\ref{prop:ln-implies-brualdi}. Yet, it also seems hard to
prove a subquadratic bound on $\ell(n)$, or even to prove
Conjecture~\ref{conj:ln}. For $\ell_r(n)$ we know more. Thanks to
\cite{BaratWanless14} and Theorem~\ref{th:rowLatin-bound}, we know
that $\frac12 n^2-O(n)<\ell_r(n)\le \big\lceil\frac{1}{4}({5-\sqrt{5}})
n^2\big\rceil$.

\subsection*{Acknowledgement}
The authors are grateful to Gwene\"al Joret and J\'anos Bar\'at for
interesting discussions.

\bibliographystyle{plain}

\begin{thebibliography}{10}

\bibitem{AkbariAlipour04}
Saieed Akbari and Alireza Alipour.
\newblock Transversals and multicolored matchings.
\newblock {\em J. Combin. Des.}, 12(5):325--332, 2004.

\bibitem{BaratWanless14}
J{\'a}nos Bar{\'a}t and Ian~M. Wanless.
\newblock Rainbow matchings and transversals.
\newblock {\em Australas. J. Combin.}, 59:211--217, 2014.

\bibitem{CW17}
Nicholas Cavenagh and Ian~M. Wanless.
\newblock Latin squares with no transversals.
\newblock arXiv:1609.03001 [math.CO].

\bibitem{Dri98}
Arthur~A. Drisko.
\newblock Transversals in row-{L}atin rectangles.
\newblock {\em J. Combin. Theory Ser. A}, 84(2):181--195, 1998.

\bibitem{EW12}
Judith Egan and Ian~M. Wanless.
\newblock {L}atin squares with restricted transversals.
\newblock {\em J. Combin. Des.}, 20(7):344--361, 2012.

\bibitem{EW16}
Judith Egan and Ian~M. Wanless.
\newblock Enumeration of {MOLS} of small order.
\newblock {\em Math. Comp.}, 85(298):799--824, 2016.

\bibitem{erdosspencer1991}
Paul Erd{\H{o}}s and Joel Spencer.
\newblock Lopsided {L}ov\'asz local lemma and {L}atin transversals.
\newblock {\em Discrete Appl. Math.}, 30(2-3):151--154, 1991.

\bibitem{HS08}
Pooya Hatami and Peter~W. Shor.
\newblock A lower bound for the length of a partial transversal in a {L}atin
  square.
\newblock {\em J. Combin. Theory Ser. A}, 115(7):1103--1113, 2008.

\bibitem{kolipaka2012sharper}
Kashyap Kolipaka, Mario Szegedy, and Yixin Xu.
\newblock A sharper local lemma with improved applications.
\newblock In {\em Approximation, randomization, and combinatorial
  optimization}, volume 7408 of {\em Lecture Notes in Comput. Sci.}, pages
  603--614. Springer, Heidelberg, 2012.

\bibitem{MMM07}
Brendan~D. McKay, Alison Meynert, and Wendy Myrvold.
\newblock Small {L}atin squares, quasigroups, and loops.
\newblock {\em J. Combin. Des.}, 15(2):98--119, 2007.

\bibitem{Wan11}
Ian~M. Wanless.
\newblock Transversals in {L}atin squares: a survey.
\newblock In {\em Surveys in combinatorics 2011}, volume 392 of {\em London
  Math. Soc. Lecture Note Ser.}, pages 403--437. Cambridge Univ. Press,
  Cambridge, 2011.

\bibitem{Woolbright78}
David~E. Woolbright.
\newblock An {$n\times n$} {L}atin square has a transversal with at least
  {$n-\surd n$} distinct symbols.
\newblock {\em J. Combin. Theory Ser. A}, 24(2):235--237, 1978.

\end{thebibliography}

\end{document}